\documentclass[11pt,a4paper]{article}
\usepackage[T1]{fontenc}
\usepackage[utf8]{inputenc}
\usepackage{lmodern}
\usepackage[a4paper,margin=27mm]{geometry}
\usepackage{amsmath,amssymb,amsthm,mathtools}
\usepackage{microtype}
\usepackage{etoolbox}
\usepackage[hidelinks,bookmarksnumbered]{hyperref}
\usepackage{xurl}
\numberwithin{equation}{section}
\newtheorem{theorem}{Theorem}[section]
\newtheorem{proposition}[theorem]{Proposition}
\newtheorem{lemma}[theorem]{Lemma}

\theoremstyle{remark}

\newcommand{\R}{\mathbb R}
\newcommand{\M}{\mathcal M}
\newcommand{\D}{\mathcal D}
\newcommand{\eps}{\varepsilon}
\newcommand{\one}{\boldsymbol{1}}
\newcommand{\diag}{\operatorname{diag}}

\newcommand{\dd}{\,\mathrm d}
\newcommand{\ip}[2]{\langle #1,#2\rangle}
\newcommand{\norm}[1]{\lVert#1\rVert}

\DeclareMathOperator{\diam}{diam}

\allowdisplaybreaks[2]
\newcommand{\authoremail}{}
\hypersetup{
  pdftitle={Uniqueness of convex central configurations in the planar (1+4)-body problem},
  pdfauthor={Kaitai Xiao},
  pdfkeywords={central configuration, dominant mass, coorbital problem, uniqueness}}
\title{Uniqueness of convex central configurations\\
in the planar $(1+4)$-body problem}
\author{Kaitai Xiao\\[3pt]
\small Chern Institute of Mathematics \& LPMC, Nankai University\\
\small Tianjin 300071, P.\ R.\ China
\ifdefempty{\authoremail}{}{\\
\small Corresponding author: \texttt{\authoremail}}}
\date{}
\begin{document}
\maketitle

\begin{abstract}
We study strictly convex central configurations of the planar Newtonian five-body problem with one dominant mass and four satellites of arbitrary positive masses. We prove that a single positive bound on the total satellite-to-primary mass ratio guarantees existence and uniqueness for every prescribed directed order of the five bodies on the convex hull, up to translations, rotations and positive scaling. The bound is independent of all satellite mass ratios, including ratios arbitrarily close to the boundary of the mass simplex. The constrained Hessian is nondegenerate on the six-dimensional shape space. We first prove uniqueness and nondegeneracy for the convex coorbital limit by an angular Hessian estimate. We then classify the possible collision limits as the satellite mass ratios vary and construct regularized local continuations for pairs, triples and pairs within triples. An exact factorization of the signed areas preserves strict convexity on independent parameter neighborhoods. These continuations connect every sufficiently small-mass configuration to the unique coorbital configuration and yield the common bound.
\end{abstract}

\section{Introduction}\label{sec:intro}

The planar $(1+N)$-body problem describes the coorbital limit of central configurations with one dominant mass~\cite{CLN}. Its geometry and stability have been studied in~\cite{Moeckel94,RS}; results for four satellites include the equal-mass classification~\cite{AF,CLO}, stacked configurations~\cite{OC}, and mass symmetries~\cite{DHW}. For three satellites, Corbera, Cors and Llibre~\cite{CCL} and Corbera, Cors, Llibre and Moeckel~\cite{CCLM} describe the bifurcations of relative equilibria. In particular, \cite[Theorem~2]{CCLM} gives convex uniqueness for every fixed positive mass ratio and sufficiently small total satellite mass.

We prove the corresponding uniqueness statement for four satellites with a bound independent of their mass ratios. Write
\begin{equation}\label{eq:masses}
 m_0=1,\qquad m_i=\eps\mu_i\quad(1\le i\le4),\qquad
 \M=\left\{\mu\in\R_{>0}^4:\sum_{i=1}^4\mu_i=1\right\}.
\end{equation}

\begin{theorem}\label{thm:main}
There is $\delta_0>0$ such that, for every $\mu\in\M$, every $0<\eps<\delta_0$, and every directed counterclockwise hull order $(0,\ell_1,\ell_2,\ell_3,\ell_4)$, the masses \eqref{eq:masses} admit exactly one strictly convex central configuration of that order, modulo translations, rotations and positive scaling. Its constrained Hessian on the six-dimensional shape space is nondegenerate.
\end{theorem}

Here $(\ell_1,\ldots,\ell_4)$ is a permutation of $(1,2,3,4)$; reflections are not included in the equivalence relation. For unnormalized masses the condition is $\sum_{i=1}^4m_i/m_0<\delta_0$.

The proof first establishes convex coorbital uniqueness. Uniformity then requires control when some $\mu_i$ tend to zero and satellites coalesce. Following the cluster approach of~\cite{Moeckel}, we classify these limits, regularize their equations, and connect every small-mass convex configuration to its coorbital limit. The signed-area estimates ensure that this continuation remains convex without imposing a relation between the mass parameter and the collision scales.

\section{Convex coorbital configurations}\label{sec:coorbital}

Put $e(t)=(\cos t,\sin t)$ and $e^\perp(t)=(-\sin t,\cos t)$. In the standard hull order, the coorbital configurations are parametrized by
\begin{equation}\label{eq:angle-domain}
 \theta_1=0,\qquad
 (\theta_2,\theta_3,\theta_4)\in\D
 =\{0<\theta_2<\theta_3<\theta_4<\pi\}.
\end{equation}
Indeed, the origin and four distinct unit-circle points form a strictly convex pentagon exactly when the points lie in an open semicircle.

We use $V$ and $A$ as in~\cite[Section~2]{CCLM}. For $0<t<\pi$, set
\begin{equation}\label{eq:kernel}
 f(t)=\sin t\bigl((2\sin(t/2))^{-3}-1\bigr),\qquad
 v(t)=\frac1{2\sin(t/2)}+1-\cos t,\qquad \kappa=-f'=v''.
\end{equation}
Extend $f$ oddly. With $r_{ij}(\theta)=|e(\theta_i)-e(\theta_j)|$, the equations are
\begin{equation}\label{eq:co-equations}
 \begin{aligned}
 E_i&=\sum_{j\ne i}\mu_jf(\theta_j-\theta_i)=0,\\
 V&=\sum_{i<j}\mu_i\mu_j
 \left(\frac1{r_{ij}(\theta)}+\frac{r_{ij}(\theta)^2}{2}\right)\\
 &=\sum_{i<j}\mu_i\mu_jv(\theta_j-\theta_i).
 \end{aligned}
\end{equation}
Thus $E=A(\theta)\mu$, where $A_{ii}=0$ and
$A_{ij}=f(\theta_j-\theta_i)$, and $\partial_{\theta_i}V=\mu_iE_i$.
Write $h=\pi/3$. Differentiation gives
\begin{equation}\label{eq:kappa}
 \begin{split}
 \kappa(t)&=\frac{2-\sin^2(t/2)}{8\sin^3(t/2)}+\cos t,\\
 \kappa'(t)&=-\frac{\cos(t/2)}{16\sin^4(t/2)}
 \bigl(6-\sin^2(t/2)+32\sin^5(t/2)\bigr)<0.
 \end{split}
\end{equation}
We shall use $f(h)=f(\pi)=0$, $f>0$ on $(0,h)$, $f<0$ on $(h,\pi)$,
$\kappa(h)=9/4$, $f\ge-1$, and $tf(t)\to\infty$ as $t\downarrow0$.

\[
 H=\left(\frac{\partial^2V}{\partial\theta_i\,\partial\theta_j}\right)_{i,j=1}^4,
 \qquad c_{ij}=c_{ji}=\mu_i\mu_j\kappa(t_{ij})\quad(i<j).
\]
\[
 H_{ii}=\sum_{j\ne i}c_{ij},\qquad H_{ij}=-c_{ij}\quad(i\ne j).
\]
Consequently, for every $u\in\R^4$,
\[
 \begin{aligned}
 u^THu
 &=\sum_{i<j}c_{ij}(u_i^2+u_j^2-2u_iu_j)\\
 &=\sum_{i<j}\mu_i\mu_j\kappa(t_{ij})(u_j-u_i)^2.
 \end{aligned}
\]
The Hessian on the slice $\theta_1=0$ is obtained by deleting the first
row and column of $H$; its quadratic form is the expression above with
$u_1=0$.

\begin{lemma}\label{lem:cycle}
Suppose $2\le N\le4$ ordered points in an open semicircle carry positive masses $w_i$ and satisfy the angular equations. Put
\[
 \ell(t)=
 \begin{cases}
 f(t)/(h-t),&t<h,\\
 0,&t=h,\\
 f(t)/(\pi-t),&t>h,
 \end{cases}
 \qquad b(t)=\kappa(t)-\ell(t).
\]
Then $b(t)>0$, and, with $t_{ij}=\theta_j-\theta_i$ for $i<j$,
\begin{equation}\label{eq:cycle-bound}
 \sum_{i<j}w_iw_j\kappa(t_{ij})(u_j-u_i)^2
 \ge\sum_{i<j}w_iw_jb(t_{ij})(u_j-u_i)^2.
\end{equation}
In particular, the angular Hessian has kernel $\R\one$.
\end{lemma}
\begin{proof}
Since $\kappa$ decreases strictly,
\[
 f(t)=\int_t^h\kappa(s)\dd s<(h-t)\kappa(t)\quad(t<h),\qquad
 f(t)=\int_t^\pi\kappa(s)\dd s<(\pi-t)\kappa(t)\quad(t>h).
\]
Together with $\kappa(h)>0$, this proves $b>0$.

Orient an edge $i<j$ rightwards if $t_{ij}<h$ and leftwards if $t_{ij}>h$, with flow $w_iw_j|f(t_{ij})|$; omit zero edges. The equations express balance at every vertex, so the flow decomposes into positive flows on simple directed cycles. Such a cycle has at most four edges. Two leftward edges would have total length $>2h$, whereas the remaining rightward edges have total length $<2h$. Hence there is one leftward edge and $k=2$ or $3$ rightward edges.

Let the rightward lengths and increments be $t_r,u_r$, with
$T=\sum_rt_r\in(h,\pi)$. The contribution of this cycle to the form with coefficient $\ell$ is its flow times
\[
 \sum_{r=1}^k\frac{u_r^2}{h-t_r}
 -\frac{(\sum_ru_r)^2}{\pi-T}
 \ge\frac{(\sum_ru_r)^2}{kh-T}
 -\frac{(\sum_ru_r)^2}{\pi-T}\ge0,
\]
because $0<kh-T\le\pi-T$. Summing over cycles proves \eqref{eq:cycle-bound}. Its right side is positive for every nonconstant vector, while rotation invariance supplies the constant kernel.
\end{proof}

\begin{theorem}\label{thm:coorbital}
For each $\mu\in\M$, equations \eqref{eq:co-equations} have exactly one root $\theta^0(\mu)$ in \eqref{eq:angle-domain}. The Hessian $D^2V$ is positive definite on the slice $\theta_1=0$, and $\theta^0$ depends smoothly on $\mu$.
\end{theorem}
\begin{proof}
The Hessian form is the left side of \eqref{eq:cycle-bound}, so the lemma gives positivity. We prove properness of the projection of the zero set to $\M$. Let $a,b,c$ be the consecutive gaps and $L=a+b+c$. Summing the mass-weighted equations over the first $k$ labels gives
\begin{equation}\label{eq:co-cut}
 \sum_{i\le k<j}\mu_i\mu_jf(\theta_j-\theta_i)=0,\qquad k=1,2,3.
\end{equation}
The signs of $f$ imply $a,b,c<h<L$. For $\mu$ in a compact subset of $\M$, let $\alpha>0$ be a common lower bound on the masses. Pairing terms in $\sum_i\mu_i\theta_iE_i=0$ yields
\[
 \sum_{i<j}\mu_i\mu_jt_{ij}f(t_{ij})=0.
\]
The negative terms sum to at least $-\pi/2$, so every consecutive gap $d$ satisfies $\alpha^2df(d)\le\pi/2$. Collisions are therefore excluded. If $L\to\pi$, the gaps tend to $h$, and the first equation tends to $\mu_3f(2h)=0$, a contradiction. This proves properness.

Since $\sum_i\mu_iE_i=0$, the last three equations recover the first.
Their Jacobian in $(\theta_2,\theta_3,\theta_4)$ is invertible by the Hessian estimate. The mass projection is thus a proper local diffeomorphism. Its nonempty image is open and closed in the connected simplex, and it is a covering with finitely many sheets; cf.~\cite[proof of Proposition~3]{CCLM}. By the classification in~\cite[main theorem]{AF}, exactly one equal-mass configuration has its satellites in an open semicircle, and its occupied gaps are symmetric under reversal. The equal-mass fiber has one point, so every fiber has one point. Local inversion gives smoothness.
\end{proof}

At distinct satellite positions, we retain the divided equations $E_i=0$ when $\mu_i=0$.
We first record the two-source probe estimate used below; cf.~\cite{DHW}.

\begin{lemma}\label{lem:probe}
Let $p,q\ge0$, $p+q=1$. If $0,h,x$ are distinct, have span less than $\pi$, and
$pf(x)+qf(x-h)=0$, then
\[
 p\kappa(|x|)+q\kappa(|x-h|)>0.
\]
\end{lemma}
\begin{proof}
For $pq>0$, put $g(x)=pf(x)+qf(x-h)$. On $(0,h)$, $g'<0$ because $\kappa>0$ there. On $(h,\pi)$, $\kappa'<0$ gives $g''>0$, while $g(\pi)=qf(\pi-h)<0$; at a root,
$g'(x)<g(\pi)/(\pi-x)<0$. Reflection $x\mapsto h-x$ treats $(h-\pi,0)$.
If $pq=0$, the claimed value is $\kappa(h)=9/4$.
\end{proof}

\begin{proposition}\label{prop:closed}
Suppose $\mu_i\ge0$, $\sum_i\mu_i=1$, four distinct angles lie in an open semicircle, and all $E_i$ vanish. Then $J=D_\theta E$ has kernel $\R\one$.
Deleting a row $p$ with $\mu_p>0$ and any column gives an invertible $3\times3$ matrix. Its inverse is uniformly bounded on closed sets of such roots with consecutive gaps bounded below, span bounded away from $\pi$, and $\mu_p\ge1/4$.
\end{proposition}
\begin{proof}
Let $P,Z$ be the positive and zero supports. One positive mass would allow probes only at the two angles $\theta_p\pm h$, so $|P|\ge2$. Lemma~\ref{lem:cycle} gives the constant kernel on $P$. For $k\in Z$, set
$D_k=\sum_{j\in P}\mu_j\kappa(|\theta_j-\theta_k|)$.
If $|P|=2$, the support separation is $h$, and Lemma~\ref{lem:probe} gives $D_k>0$.

If $|P|=3$, the nonzero skew matrix $A_{PP}$ has rank two. Since $A\mu=0$, the $4\times4$ skew matrix $A$ also has rank two. Its two-dimensional kernel cannot consist only of vectors vanishing at the unique label $k\in Z$, since restriction would inject it into $\ker A_{PP}$. Choose $\nu\in\ker A$ with $\nu_k>0$. For small $s>0$, $\mu+s\nu$ is positive and solves the same angular equations. Apply \eqref{eq:cycle-bound} to the $k$th coordinate vector, divide by $s\nu_k$, and let $s\downarrow0$:
\[
 D_k\ge\sum_{j\in P}\mu_j b(|\theta_j-\theta_k|)>0.
\]
In support order,
\[
 J=\begin{pmatrix}J_{PP}&0\\J_{ZP}&\diag(D_k:k\in Z)\end{pmatrix}.
\]
Every kernel vector is therefore constant. Also $\mu^TJ=0$ and $J\one=0$.
If one coordinate of a vector is zero and the retained rows annihilate it, the identity $\mu^TJ=0$ recovers the omitted row; the vector is constant and hence zero. The inverse bound follows from compactness and continuity.
\end{proof}

\section{Full equations and local continuation}\label{sec:variation}

We adopt the centered normalization and notation of~\cite[Section~2]{CCLM}:
\begin{equation}\label{eq:cc}
 \sum_{j\ne i}m_j\frac{x_j-x_i}{r_{ij}^3}=-x_i,\qquad
 \sum_{i=0}^4m_ix_i=0,\qquad r_{ij}=|x_i-x_j|,
\end{equation}
and
\begin{equation}\label{eq:UI}
 U=\sum_{i<j}\frac{m_im_j}{r_{ij}},\qquad
 I=\sum_{i=0}^4m_i|x_i|^2,\qquad G=U+\tfrac12 I.
\end{equation}
Thus $x_i=r_ie(\theta_i)$ denotes a centered polar coordinate. For the collision analysis use the primary-relative coordinates
\[
 q_i=x_i-x_0=\rho_i e(\varphi_i),\qquad
 X=\sum_{i=1}^4\mu_iq_i,\qquad
 Y=\sum_{i=1}^4\mu_iq_i/\rho_i^3,\qquad c_\eps=(1+\eps)^{-1}.
\]
Then $q_0=0$, $x_0=-\eps c_\eps X$, $x_i=q_i-\eps c_\eps X$, and $r_{ij}=|q_i-q_j|$.
At $\eps=0$ the centered and relative coordinates coincide. Below, sums over satellite labels run from $1$ to $4$.

Subtracting the primary equation gives
\begin{equation}\label{eq:G}
 F_i=F_0(q_i)+\eps\sum_{j\ne i}\mu_j\frac{q_j-q_i}{r_{ij}^3}-\eps Y=0,
 \qquad F_0(q)=q-\frac q{|q|^3}.
\end{equation}
These differences recover \eqref{eq:cc}, since all acceleration residuals are then equal and their mass-weighted sum is zero. In relative coordinates,
\begin{equation}\label{eq:action}
 \Phi=\frac G\eps
 =\sum_i\mu_i\left(\rho_i^{-1}+\frac{\rho_i^2}{2}\right)
 +\eps\sum_{i<j}\frac{\mu_i\mu_j}{r_{ij}}
 -\frac{\eps c_\eps}{2}|X|^2.
\end{equation}
Indeed $I=\eps\sum_i\mu_i\rho_i^2-\eps^2c_\eps|X|^2$.
For positive masses, $\nabla_q\Phi=(C\otimes I_2)F$, where
$C=\diag(\mu_i)-\eps c_\eps\mu\mu^T$ is positive definite for $\eps>-1$.
Thus $F=0$ is equivalent to $\nabla\Phi=0$.

Put $g(\rho)=\rho-\rho^{-2}$ and $\Delta_{ij}=\varphi_j-\varphi_i$.
The mass-divided radial and angular residuals are
\begin{align}
 R_i&=g(\rho_i)-\eps\sum_{j\ne i}\mu_j
       \frac{\rho_i-\rho_j\cos\Delta_{ij}}{r_{ij}^3}
       -\eps c_\eps X\cdot e(\varphi_i),\label{eq:R}\\
 T_i&=\sum_{j\ne i}\mu_j\rho_i\rho_j\sin\Delta_{ij}
       (r_{ij}^{-3}-c_\eps).\label{eq:T}
\end{align}
They satisfy, also away from solutions,
\begin{equation}\label{eq:polar-gradient}
 \partial_{\rho_i}\Phi=\mu_iR_i,\qquad
 \partial_{\varphi_i}\Phi=\eps\mu_iT_i,\qquad
 \sum_i\mu_iT_i=0.
\end{equation}
For $\eps>0$ and positive masses, retain four radial rows on a rotation slice and omit one angular row. The resulting seven equations are equivalent to \eqref{eq:cc}. Their explicit divided rows remain meaningful at zero masses.

For estimates we also use the tangential projection of \eqref{eq:G}:
\begin{equation}\label{eq:Q}
 Q_i=\sum_{j\ne i}\mu_j\rho_j\sin\Delta_{ij}
       (r_{ij}^{-3}-\rho_j^{-3})=0.
\end{equation}
At solutions, $\sum_i\mu_iF_i=X-(1+\eps)Y=0$, so $Y=c_\eps X$ and
$T_i=\rho_iQ_i$.

\begin{lemma}\label{lem:hessian}
For positive masses and $\eps>0$, the seven-equation Jacobian on a rotation slice is invertible if and only if the constrained shape Hessian is nondegenerate. Invertible coordinate changes and row operations preserve this equivalence.
\end{lemma}
\begin{proof}
When the omitted row and fixed angle have the same label, \eqref{eq:polar-gradient} identifies the slice Hessian of $\Phi$ with the residual Jacobian multiplied by a positive diagonal matrix. Any other positive-mass row can be omitted using $\sum_i\mu_iT_i=0$. At a root, an invertible coordinate map $\Psi$ and row map $L$ give
$D(L\nabla\Phi\circ\Psi)=L(D^2\Phi)D\Psi$.

Let $M=\diag(m_0I_2,\ldots,m_4I_2)$,
$E=\{v:\sum_i m_iv_i=0\}$, and $T=\{v\in E:\ip{Mx}{v}=0\}$.
The constrained Hessian is $D^2G|_T$ modulo the rotation direction; see~\cite{MoeckelCC}. Homogeneity gives
\[
 D^2G(x)x=3Mx,\qquad E=\R x\oplus T,
\]
an orthogonal splitting for $D^2G$ with radial value $3I>0$.
Relative coordinates give a linear isomorphism onto $E$, and $G=\eps\Phi$.
Passing to a transverse rotation slice and removing this nonzero radial direction leaves the six-dimensional shape Hessian.
\end{proof}

We repeatedly use the following compact form of the implicit function theorem.

\begin{lemma}\label{lem:product}
Let $F(y;s,\beta)$ be smooth near
$\{(y^0(\beta),0,\beta):\beta\in K\}$, where $K$ is compact,
$F(y^0(\beta);0,\beta)=0$, and $D_yF$ is invertible.
If the base roots and inverse matrices depend continuously on $\beta$, then the local roots exist uniquely on uniform independent parameter neighborhoods, with a common root radius and inverse bound.
\end{lemma}
\begin{proof}
Apply the implicit function theorem locally and take a finite cover of the compact base graph.
\end{proof}

\begin{proposition}\label{prop:local}
For each $K\Subset\M$, the coorbital graph
$(1,1,1,1,\theta_2^0(\mu),\theta_3^0(\mu),\theta_4^0(\mu))$
has a unique nearby smooth solution branch for sufficiently small $|\eps|$, uniformly for $\mu$ near $K$. For $\eps>0$, it is strictly convex in the same order and has nondegenerate shape Hessian.
\end{proposition}
\begin{proof}
This is the continuation of~\cite[Proposition~1 and its proof]{CCLM}.
At $\eps=0$, the radial block is $3I_4$, its angular columns vanish, and the reduced angular block is invertible by Theorem~\ref{thm:coorbital}. Lemma~\ref{lem:product} gives compact uniformity, the limiting signed areas preserve strict convexity, and Lemma~\ref{lem:hessian} gives nondegeneracy.
\end{proof}

For the divided differences below we write
$s(t)=\sin t/t$, with its smooth value $s(0)=1$.

\section{Uniform geometry and collision limits}
\label{sec:limits}
Throughout this section, $\eps\to0$ along actual multiplier-one solutions with arbitrary positive satellite ratios. After passage to a subsequence, a complete cluster consists of all labels with the same limiting position. Constants may depend on the selected sequence, but not on its index.

\begin{proposition}\label{prop:uniform}
Every configuration satisfies $\diam(x)\le4(1+\eps)^{1/3}$. For each $E>0$ there is $b(E)>0$ such that $\rho_i\ge b(E)$ whenever $0<\eps\le E$. Uniformly over all positive ratios and all configurations, $\max_i|\rho_i-1|\to0$ as $\eps\downarrow0$.

For strictly convex configurations in standard order, there are $\eta_0,\eps_0>0$, independent of the ratios, such that
\begin{equation}\label{eq:span}
 \eta_0\le\varphi_4\le\pi-\eta_0\qquad(0<\eps<\eps_0).
\end{equation}
Each consecutive angular gap has limit superior at most $h$.
\end{proposition}
\begin{proof}
Project onto a diameter direction and split the five labels across their largest consecutive projection gap $g\ge\diam(x)/4$. If the two groups have masses $M_A,M_B$, their projected centers are separated by
\[
 g\le\Delta=\left(M_A^{-1}+M_B^{-1}\right)
 \sum_{i\in A,j\in B}m_im_j\frac{(x_j-x_i)\cdot e}{|x_j-x_i|^3}
 \le\frac{1+\eps}{g^2}.
\]
This proves the diameter bound.

If satellites approached the primary with $\eps\le E$, take the complete limiting cluster $C$ containing it. Let $d\to0$ be its diameter, $s$ its satellite mass, $b$ its center, and $I_C,U_C$ its internal inertia and potential. Then $U_C\ge s/d$ and $I_C\le sd^2$. The exterior acceleration $A$ has a uniformly bounded derivative near $C$. Summing the cluster equations against $m_i(x_i-b)$ gives
\[
 U_C=I_C+\sum_{i\in C}m_i(x_i-b)\cdot[A(x_i)-A(b)]\le(1+L)I_C.
\]
Cancellation of $s$ contradicts $d\to0$ and proves the uniform lower radius bound.

Now let $C$ be a complete satellite cluster with nonzero limit $q^0$. Average \eqref{eq:G} with the current weights $\mu_i/\sum_{j\in C}\mu_j$. Internal forces cancel, and all exterior distances stay positive. Since $Y$ is bounded, the average is $F_0(q^0)+o(1)=O(\eps)$. Thus $|q^0|=1$; a violating sequence proves uniform convergence of all radii.

Strict convexity gives $0=\varphi_1<\cdots<\varphi_4<\pi$. Summing the first $k$ weighted angular equations gives
\begin{equation}\label{eq:full-cut}
 \sum_{i\le k<j}\mu_i\mu_j\rho_i\rho_j\sin(\varphi_j-\varphi_i)(r_{ij}^{-3}-c_\eps)=0.
\end{equation}
A gap exceeding $h$ by a fixed amount makes every term negative for small $\eps$. If $\varphi_4\to\pi=3h$, all three gaps tend to $h$. But the skew matrix
\[
 B_{ij}=\rho_i\rho_j\sin(\varphi_j-\varphi_i)(r_{ij}^{-3}-c_\eps),\qquad B_{ii}=0,
\]
satisfies $B\mu=0$, whereas its limiting determinant is $f(2h)^4>0$. If $\varphi_4\to0$, all satellite distances vanish and every term of \eqref{eq:full-cut} is positive. These contradictions give \eqref{eq:span}.
\end{proof}

For the remainder of this section, all configurations are strictly convex in the standard directed order. We shall use
\begin{equation}\label{eq:distance}
 r_{ij}^2=(\rho_j-\rho_i)^2+4\rho_i\rho_j\sin^2\frac{\varphi_j-\varphi_i}{2},
\end{equation}
and the strict support inequality for any adjacent hull edge $a\to b$ and any other vertex $j$:
\begin{equation}\label{eq:support}
 \det(q_b-q_a,q_j-q_a)>0.
\end{equation}

\begin{proposition}\label{prop:pair-scale}
Let $\{a,b\}$ be a complete pair, $a<b$, and put $d=|q_b-q_a|$, $\nu=\mu_a+\mu_b$, $\sigma=1-\nu$, $\delta=\varphi_b-\varphi_a$, and $\Delta\rho=\rho_b-\rho_a$. The pair is adjacent and
\begin{equation}\label{eq:pair-scale}
 \nu\to0,\qquad \delta/d\to1,\qquad |\Delta\rho|\le C\eps\sigma d,
 \qquad \nu/d^3\to D>0.
\end{equation}
The two exterior limits are distinct and separated by $h$. In the gauge placing them at $0,h$, with limiting masses $p,1-p$ and pair angle $\psi$,
\[
 pf(\psi)+(1-p)f(\psi-h)=0,\qquad
 D=p\kappa(|\psi|)+(1-p)\kappa(|\psi-h|).
\]
Two complete pairs cannot coexist.
\end{proposition}
\begin{proof}
Equal-angle classes are consecutive blocks. The exterior fields
\[
 P(q)=\sum_{j\notin\{a,b\}}\mu_j\frac{q_j-q}{|q_j-q|^3},\qquad
 H(\rho,v)=\sum_{j\notin\{a,b\}}\mu_j\rho_j\sin(\varphi_j-v)
 (|\rho e(v)-q_j|^{-3}-\rho_j^{-3})
\]
have first derivatives $O(\sigma)$ near the pair. Subtraction of \eqref{eq:Q} gives
\begin{equation}\label{eq:pair-qdiff}
 H(\rho_b,\varphi_b)-H(\rho_a,\varphi_a)
 =\sin\delta\left(\frac{\mu_a\rho_a+\mu_b\rho_b}{d^3}
 -\frac{\mu_a}{\rho_a^2}-\frac{\mu_b}{\rho_b^2}\right)>0.
\end{equation}
Since $\delta\le Cd$, this implies
\begin{equation}\label{eq:pair-prebound}
 \delta\nu/d^3\le C\sigma d.
\end{equation}
The vector difference is
\begin{equation}\label{eq:pair-vdiff}
 F_0(q_b)-F_0(q_a)+\eps[P(q_b)-P(q_a)]=\lambda(q_b-q_a),
 \qquad \lambda=\eps\nu/d^3.
\end{equation}
Its integral derivative bounds $\lambda$. Since $DF_0\to\diag(3,0)$ in the radial--tangential basis, a limiting unit pair direction is radial with $\lambda\to3$, or positively tangential with $\lambda\to0$.

In a radial subsequence, $\delta=o(d)$ and $|\Delta\rho|\sim d$. The support inequality expands as
\[
 \det(q_b-q_a,q_j-q_a)=\Delta\rho\,\rho_j\sin(\varphi_j-\varphi_a)+o(d).
\]
An interior pair would require opposite signs of $\Delta\rho$. For a left endpoint pair it requires $\Delta\rho>0$, whereas on its polar rectangle
\[
 H_\rho=-3\sum_{j\notin\{a,b\}}\mu_j\rho_j\sin(\varphi_j-v)
 [\rho-\rho_j\cos(\varphi_j-v)]|\rho e(v)-q_j|^{-5}\le-\sigma/C.
\]
Thus the left side of \eqref{eq:pair-qdiff} is at most $-\sigma|\Delta\rho|/C+C\sigma\delta<0$. At the right endpoint both radial signs reverse, giving the same contradiction.

Consequently $\delta/d\to1$, $\Delta\rho/d\to0$, and \eqref{eq:pair-prebound} yields $\nu/d^3\le C\sigma$, hence $\nu\to0$. Projecting \eqref{eq:pair-vdiff} onto the midpoint ray gives exactly
\begin{equation}\label{eq:midpoint}
 [g(\rho_b)-g(\rho_a)-\lambda\Delta\rho]\cos(\delta/2)
 =-\eps[P(q_b)-P(q_a)]\cdot e((\varphi_a+\varphi_b)/2).
\end{equation}
Since $g'\to3$, $\lambda\to0$, and $DP=O(\sigma)$, this proves the radial bound.

A second complete pair would also have vanishing total mass, which is impossible. The exterior angular equations therefore give separation $h$. The pair's weighted $T$-average, divided by $\nu$ after internal cancellation, gives the probe equation. Finally divide \eqref{eq:pair-qdiff} by $\delta$: its radial contribution tends to zero, its angular derivative tends to $D$, and its right side is $\nu/d^3+o(1)$. Lemma~\ref{lem:probe} gives $D>0$.
\end{proof}

\begin{proposition}\label{prop:triple-scale}
Let $E$ be a complete triple with exterior label $p$, $\nu=\sum_{i\in E}\mu_i$, $\sigma=1-\nu$, diameter $d$, and angular width $W$. It is an endpoint block and
\begin{equation}\label{eq:triple-scale}
 \begin{gathered}
 \nu\to0,\quad\sigma\to1,\quad|\varphi_p-\varphi_i|\to h\quad(i\in E),\\
 \max_{i,j\in E}|\rho_i-\rho_j|=o(d),\qquad W/d\to1,\qquad
 0<C^{-1}\le\nu/d^3\le C\sigma.
 \end{gathered}
\end{equation}
\end{proposition}
\begin{proof}
It suffices to treat $E=\{1,2,3\}$, $p=4$. Put
\[
 J(\rho,v)=\rho\rho_p\sin(\varphi_p-v)(|\rho e(v)-q_p|^{-3}-c_\eps),\qquad
 J_i=J(\rho_i,\varphi_i).
\]
Summing the weighted triple equations and cancelling $\sigma\nu$ gives $f(\varphi_p^0-\psi)=0$ at the limiting center $\psi$. Separation and \eqref{eq:span} force $\varphi_p^0-\psi=h$. On the triple's polar rectangle,
\begin{equation}\label{eq:triple-J}
 J_\rho\to-3\sqrt3/4,\qquad J_v\to D_0:=9/4.
\end{equation}
Expanding \eqref{eq:support} against the exterior gives $\rho_j-\rho_i\ge-C(\varphi_j-\varphi_i)$ for adjacent, hence all ordered, triple pairs. The first and last angular equations give $J_1<0<J_3$. If $\rho_3\ge\rho_1$, integration of \eqref{eq:triple-J} yields
\[
 0<J_3-J_1\le-(\rho_3-\rho_1)/C+CW.
\]
The opposite case uses the support bound. All intermediate radii lie between $\rho_1-CW$ and $\rho_3+CW$, so the radial spread is $O(W)$ and $d\asymp W$.

For internal pairs, $k_{ij}=\rho_i\rho_j s(\varphi_j-\varphi_i)(r_{ij}^{-3}-c_\eps)\ge C^{-1}d^{-3}$. Subtracting the first and last angular equations, each internal mass multiplies angular distances summing to $W$. Hence
\[
 \nu W/d^3\le C\sigma(J_3-J_1)\le C\sigma W.
\]
Thus $\nu/d^3\le C\sigma$, $\nu\to0$, and $\sigma\to1$.

Along any further subsequence, extract limits of $(q_i-q_1)/d$ and group equal limits. There are at least two groups. Average \eqref{eq:G} within each group with its current normalized masses. Internal forces cancel; other groups contribute at most $C\nu/d^2=O(d)$ before multiplication by $\eps$. Each group has diameter $o(d)$, while the exterior field has bounded derivative. Subtracting the averages gives, for representatives,
\[
 F_0(q_i)-F_0(q_j)=o(d)+O(\eps d).
\]
Since $DF_0\to3e(\psi)e(\psi)^T$, radial differences between groups are $o(d)$, and the same holds within groups. Equation~\eqref{eq:distance} now gives $W/d\to1$.

If $\nu/d^3\to0$, repeat these normalized averages in the angular equations. Cross-group interactions are $O(\nu/d^2)=o(d)$, and each averaged $J$ differs from its representative value by $o(d)$. Subtraction, \eqref{eq:triple-J}, and the radial estimate give $D_0(\varphi_i-\varphi_j)=o(d)$, contradicting distinct scaled groups. This proves the lower bound. Reflection gives the other endpoint block.
\end{proof}

\begin{proposition}\label{prop:nested-scale}
Suppose the shortest internal distance $s$ in Proposition~\ref{prop:triple-scale} satisfies $s/d\to0$. Choose fixed shortest-pair labels $a<b$, and let $c$ be the third label. Put $\xi=\mu_a+\mu_b$, $\delta=\varphi_b-\varphi_a$, and $\Delta\rho=\rho_b-\rho_a$. The pair is adjacent and
\begin{equation}\label{eq:nested-scale}
 \xi/\nu\to0,\qquad \nu/d^3\to D_0,\qquad \xi/s^3\to3D_0,
 \qquad \delta/s\to1,\qquad |\Delta\rho|\le C\eps s.
\end{equation}
These limits do not require convergence of $\mu_a/\xi$.
\end{proposition}
\begin{proof}
The shortest pair is adjacent: otherwise $W=\delta\le Cs=o(d)$. Let $\zeta=1$ for a left triple and $-1$ for a right triple, and $\omega=\operatorname{sign}(\varphi_c-\varphi_a)$. The distances from the pair to $c$ are $d(1+o(1))$; their angular differences divided by $d$ tend to $\omega$, and their radial differences are $o(d)$. These estimates hold throughout the pair's segment and polar rectangle.

Define $P,H$ as in the pair proof, now summing over $j=c,p$, and put $D_j=|\rho e(v)-q_j|$, $\Delta_j=\varphi_j-v$. Since $\mu_c/d^3\le C$, $DP=O(1)$. Differentiation gives
\begin{align}
 H_\rho&=-3\sum_{j=c,p}\mu_j\rho_j\sin\Delta_j(\rho-\rho_j\cos\Delta_j)D_j^{-5},\label{eq:nested-field}\\
 H_v&=\sum_{j=c,p}\mu_j\rho_j[-\cos\Delta_j(D_j^{-3}-\rho_j^{-3})+3\rho\rho_j\sin^2\Delta_jD_j^{-5}].\notag
\end{align}
The near-source part of $H_\rho$ is bounded by
$C(\mu_c/d^3)(|\rho-\rho_c|/d+d)=o(1)$, while its contribution to $H_v$ is $2\mu_c/d^3+o(1)$. The exterior contributes $-\zeta3\sqrt3/4$ and $D_0$, respectively. Thus
\begin{equation}\label{eq:nested-derivatives}
 \zeta H_\rho\le-C^{-1},\qquad |H_\rho|+|H_v|\le C,
 \qquad H_v=D_0+2\mu_c/d^3+o(1).
\end{equation}
The vector difference \eqref{eq:pair-vdiff}, with $d,\nu$ replaced by $s,\xi$, again has only radial or tangential limits. In the radial case, support against $p$ forces $\operatorname{sign}(\Delta\rho)=\zeta$. Since $\delta=o(s)$ and $|\Delta\rho|\sim s$, \eqref{eq:nested-derivatives} makes the left side of \eqref{eq:pair-qdiff} negative, contradicting its positive right side. Hence $\lambda\to0$, $\delta/s\to1$, and $\Delta\rho/s\to0$. The angular difference gives $\xi/s^3\le C$, whence
\[
 \frac\xi\nu=\frac{\xi/s^3}{\nu/d^3}\left(\frac sd\right)^3\to0.
\]
The midpoint identity \eqref{eq:midpoint}, with $DP=O(1)$, gives $|\Delta\rho|\le C\eps s$.

Average the pair's $T$ equations with weights $\mu_a/\xi,\mu_b/\xi$ and subtract the equation of $c$. Let $K$ be the averaged variational kernel from the pair to $c$, and $J_0$ its averaged exterior field. Internal cancellation gives exactly
\begin{equation}\label{eq:nested-macro-average}
 0=\nu K+\sigma(J_0-J_c).
\end{equation}
The triple estimates imply $K=\omega d^{-2}(1+o(1))$ and $J_0-J_c=-D_0\omega d+o(d)$. Division by $\omega d$ gives $\nu/d^3\to D_0$, hence $\mu_c/d^3\to D_0$. Now $H_v\to3D_0$ in \eqref{eq:nested-derivatives}; dividing the pair angular difference by $\delta$, with $\Delta\rho/\delta=O(\eps)$, gives $\xi/s^3\to3D_0$. All averages use current normalized masses, so the conclusions hold on the selected fixed-label sequence.
\end{proof}

Every sequence therefore has, after extraction, one of four types: no satellite collision, a complete pair, a triple with comparable internal distances, or a triple containing a smaller pair. The span bound excludes a four-member cluster, and Proposition~\ref{prop:pair-scale} excludes two pairs.

\section{Regularized solution charts}
\label{sec:charts}
We use the residuals \eqref{eq:R}--\eqref{eq:T}, changing only coordinates and rows. Put
\begin{equation}\label{eq:chi}
 \chi(t,z)=\tfrac12tz^2s(tz/2)^2,\qquad 1-\cos(tz)=t\chi(t,z).
\end{equation}
The extensions to signed scales below are analytic extensions of the positive-scale equations. For a smooth target field $K$,
\begin{equation}\label{eq:hadamard}
 \frac{K(\rho+tu,v+tz)-K(\rho,v)}t
 =\int_0^1(uK_\rho+zK_v)(\rho+\lambda tu,v+\lambda tz)\dd\lambda.
\end{equation}
Here source positions and the common vector $X$ are fixed in the target difference; subsequent differentiation includes their dependence on the variables.

One formula covers the internal terms of both single-scale charts. If $\mu_i=t^3\alpha_i$, $\rho_i=R+tw_i$, $\varphi_i=\psi+tz_i$, set $a_{ij}=w_i-w_j$, $b_{ij}=z_i-z_j$ and
\[
 W_{ij}^2=a_{ij}^2+\rho_i\rho_jb_{ij}^2s(tb_{ij}/2)^2,
 \qquad r_{ij}=tW_{ij}\quad(t>0).
\]
The internal contributions to $R_i,T_i$ are, respectively,
\begin{equation}\label{eq:regular-int}
 \begin{aligned}
 &-\eps t\sum_{j\ne i}\alpha_j[a_{ij}+\rho_j\chi(t,b_{ij})]W_{ij}^{-3},\\
 &-t\sum_{j\ne i}\alpha_j\rho_i\rho_jb_{ij}s(tb_{ij})(W_{ij}^{-3}-t^3c_\eps).
 \end{aligned}
\end{equation}
Thus their divided differences are smooth wherever $W_{ij}>0$. Their weighted angular sum vanishes, and their weighted radial sum is
\[
 -\eps\sum_{i<j}\alpha_i\alpha_j(\rho_i+\rho_j)
 [1-\cos(tb_{ij})]W_{ij}^{-3}.
\]
All exterior terms are smooth target fields and use \eqref{eq:hadamard}.

At a distinct limiting root with mass $\mu^0$ in the closed simplex, choose $p$ with $\mu_p^0\ge1/4$, fix $\varphi_p=0$ and omit $T_p$. The seven residuals are smooth for nearby signed masses of sum one. Their derivative at $\eps=0$ has the form
\[
 \begin{pmatrix}3I_4&0\\ *&K\end{pmatrix},
\]
where $K$ is the reduced derivative from Proposition~\ref{prop:closed}. Lemma~\ref{lem:product} gives a unique root on an independent product in $\eps$ and the mass coordinates. Convexity persists since all cyclic triple areas have positive limits; Lemma~\ref{lem:hessian} gives nondegeneracy at positive masses. Convergence places every sequence with these limiting data in the root neighborhood. The omitted angular equation follows from \eqref{eq:polar-gradient}.

For a complete pair, fix limiting exterior angles $0,h$, exterior mass fraction $p_0$, and pair angle $\psi_0$ as in Proposition~\ref{prop:pair-scale}. Let $D>0$ be the probe derivative there and $z_0=D^{-1/3}$. For independent parameters $(\eps,t,p,\alpha)$, put $\beta=1-\alpha$ and
\begin{gather}
 \mu_a=t^3\alpha,\quad\mu_b=t^3\beta,\quad
 \mu_c=(1-t^3)p,\quad\mu_d=(1-t^3)(1-p),\label{eq:pair-chart-mass}\\
 \varphi_c=0,\quad\varphi_d=\omega,\quad
 \rho_a=R-\beta tu,\quad\rho_b=R+\alpha tu,\quad
 \varphi_a=\psi-\beta tz,\quad\varphi_b=\psi+\alpha tz.\label{eq:pair-coordinates}
\end{gather}
Use variables $(\rho_c,\rho_d,R,u,\omega,\psi,z)$. For either row family $L=R,T$, put $A_L=\alpha L_a+\beta L_b$, $B_L=(L_b-L_a)/t$, and retain
\[
 F=(R_c,R_d,A_R,B_R,T_o,A_T,B_T),\qquad
 o=\begin{cases}d,&p_0\ge1/2,\\c,&p_0<1/2.\end{cases}
\]
The omitted exterior row has positive mass bounded away from zero. Both coordinate changes have determinant $t$; the inverse row formulas $L_a=A_L-\beta tB_L$, $L_b=A_L+\alpha tB_L$ require no division by a mass.

Here
\begin{equation}\label{eq:pair-W}
 r_{ab}^2=t^2W^2,\qquad W^2=u^2+\rho_a\rho_bz^2s(tz/2)^2.
\end{equation}
Equation~\eqref{eq:regular-int} and \eqref{eq:hadamard} extend every row smoothly near $W=z_0>0$, uniformly for $\alpha$ in an interval containing $[0,1]$. At $\eps=t=0$ the radial rows are $g(\rho_c),g(\rho_d),g(R),g'(R)u$. With $(p_c,p_d)=(p,1-p)$ and
\[
 H(\rho,v)=\rho\sum_{j=c,d}p_j\rho_j\sin(\varphi_j-v)
 (|\rho e(v)-q_j|^{-3}-1),
\]
the cluster angular rows become
\begin{equation}\label{eq:pair-boundary}
 A_T=H(R,\psi),\qquad
 B_T=H_\rho u+H_vz-\frac{R^2z}{(u^2+R^2z^2)^{3/2}}.
\end{equation}
At $y^0=(1,1,1,0,h,\psi_0,z_0)$ the radial block is $(3I_4,0)$ and the angular block is lower triangular, with diagonal
\begin{equation}\label{eq:pair-J}
 (a_0,D,3D),\qquad
 a_0=\begin{cases}D_0p_0,&o=d,\\-D_0(1-p_0),&o=c.\end{cases}
\end{equation}
Since $|a_0|\ge9/8$, the derivative is invertible. The base and derivative are independent of $\alpha$, so Lemma~\ref{lem:product} gives a common root on an independent $(\eps,t,p)$ product, uniformly in $\alpha\in[0,1]$. At $\eps=0$ row recovery gives unit radii and $u=0$; hence $u=O(\eps)$ uniformly.

Convexity follows from the areas containing the pair. Set $q_0=0$ and $S_j=\det(q_b-q_a,q_j-q_a)$. At the base point, their smooth quotients satisfy
\begin{equation}\label{eq:pair-areas}
 (S_0/t)|_0=z_0,\qquad
 (S_j/t)|_0=z_0[1-\cos(\varphi_j^0-\psi_0)]>0\quad(j=c,d).
\end{equation}
All other cyclic triple areas have positive limits. The invertible row and coordinate changes and Lemma~\ref{lem:hessian} give nondegeneracy. Finally, Proposition~\ref{prop:pair-scale} gives $t=\nu^{1/3}$, $z=\delta/t\to z_0$, $u=(\rho_b-\rho_a)/t\to0$, and, after extraction, $p\to p_0$. The actual variables therefore enter the common root ball, without convergence of $\alpha=\mu_a/\nu$.

For a triple with comparable internal distances, label its members $1,2,3$ in increasing angular order and its exterior $p$. Its limiting angle in the gauge $\varphi_p=0$ is $\tau h$, $\tau=\pm1$. Consider
\begin{equation}\label{eq:internal}
 \mathcal G_i(z;\alpha)=D_0z_i-
 \sum_{j\ne i}\alpha_j\frac{z_i-z_j}{|z_i-z_j|^3}=0,
 \qquad z_1<z_2<z_3,\quad \alpha_i\ge0,\quad\sum_i\alpha_i=1.
\end{equation}
For positive weights these are the divided gradient equations of the proper, strictly convex function
\[
 \mathcal V(z)=\frac{D_0}2\sum_i\alpha_i z_i^2+
 \sum_{i<j}\frac{\alpha_i\alpha_j}{|z_i-z_j|},
\]
so there is exactly one root and its weighted center is zero. For nonnegative weights the divided derivative $J=D_z\mathcal G$ has entries
\begin{equation}\label{eq:internal-J}
 J_{ii}=D_0+2\sum_{j\ne i}\frac{\alpha_j}{|z_i-z_j|^3},\qquad
 J_{ij}=-\frac{2\alpha_j}{|z_i-z_j|^3}\quad(i\ne j).
\end{equation}
Strict row diagonal dominance gives $\norm{Jv}_\infty\ge D_0\norm v_\infty$. The integral of $J$ along any ordered segment has the same property, proving uniqueness also for nonnegative weights. Boundary existence is needed only at limits obtained below.

Set
\begin{equation}\label{eq:LB}
 L=\begin{pmatrix}-1&1&0\\0&-1&1\end{pmatrix},\qquad
 B_\alpha=\begin{pmatrix}-(\alpha_2+\alpha_3)&-\alpha_3\\
 \alpha_1&-\alpha_3\\\alpha_1&\alpha_1+\alpha_2\end{pmatrix}.
\end{equation}
Then $LB_\alpha=I_2$, $\alpha^TB_\alpha=0$. Fix a distinct root $z^0$ at $\alpha^0$, put $v^0=Lz^0$, and use
\begin{equation}\label{eq:regular-coordinates}
 \begin{gathered}
 \mu_i=t^3\alpha_i,\quad\mu_p=1-t^3,\qquad
 w=B_\alpha u,\quad z=B_\alpha v,\\
 \rho_p=r,\quad\varphi_p=0,\quad
 \rho_i=R+tw_i,\quad\varphi_i=\psi+tz_i.
 \end{gathered}
\end{equation}
For $P=(P_1,P_2,P_3)^T$, either $R$ or $T$, retain $A_P=\alpha^TP$ and $B_P=LP/t$, together with $R_p$. The inverse $P=A_P\one+tB_\alpha B_P$ retains every probe row, including when $\alpha_i=0$; $T_p$ follows since $\mu_p>0$. Equation~\eqref{eq:regular-int} proves smoothness on an independent $(\eps,t,\alpha)$ neighborhood: all $W_{ij}$ are positive at $z^0$.

At $\eps=t=0$ the radial rows are $g(r),g(R),g'(R)u_1,g'(R)u_2$. Put
\begin{equation}\label{eq:Hzero}
 \mathcal H(\rho,v;r)=-\rho r\sin v
 \bigl((\rho^2+r^2-2\rho r\cos v)^{-3/2}-1\bigr).
\end{equation}
The angular rows are $\mathcal H(R,\psi;r)$ and $LM$, where
\[
 M_i=\mathcal H_\rho w_i+\mathcal H_vz_i-
 \sum_{j\ne i}\frac{\alpha_jR^2(z_i-z_j)}{[(w_i-w_j)^2+R^2(z_i-z_j)^2]^{3/2}}.
\]
At $(r,R,u,\psi)=(1,1,0,\tau h)$ the last rows are $L\mathcal G$. Since $\alpha^T\mathcal G=D_0\alpha^Tz=0$, they recover all of \eqref{eq:internal}. At $y^0=(1,1,0,0,\tau h,v_1^0,v_2^0)$ the radial block is $(3I_4,0)$ and the angular block has diagonal blocks $D_0$ and
\begin{equation}\label{eq:Kreg}
 K=LJB_\alpha.
\end{equation}
If $K\xi=0$, then $JB_\alpha\xi=c_1\one$. The identity $\alpha^TJ=D_0\alpha^T$ gives $c_1=0$, hence $\xi=0$. Thus Lemma~\ref{lem:product} applies. At $\eps=0$ row recovery gives $r=R=1,u=0$, including $t=0$ by continuity, so $|r-1|+|R-1|+|u|=O(\eps)$. Nondegeneracy follows from Lemma~\ref{lem:hessian}; convexity is proved in Section~\ref{sec:convexity}.

For capture, Proposition~\ref{prop:triple-scale} gives $t=\nu^{1/3}\asymp d$, $u\to0$ and bounded positive gaps $v$. Extract $\alpha\to\alpha^0$, $v\to v^0$ with both gaps positive. The smooth difference rows limit to $L\mathcal G=0$, hence \eqref{eq:internal}. The resulting boundary root supplies the chart, and local uniqueness captures the sequence.

For a triple containing a smaller pair, let $a<b$ denote that pair, $c$ the third cluster member and $p$ the exterior. Let $\tau=\pm1$ specify the cluster side relative to $p$, and $\omega=\pm1$ the side of $c$ relative to the pair. Put $\beta=1-\alpha$, $\gamma=\eta^3$, and
\begin{equation}\label{eq:nest-masses}
 \mu_a=t^3\gamma\alpha,\quad\mu_b=t^3\gamma\beta,\quad
 \mu_c=t^3(1-\gamma),\quad\mu_p=1-t^3.
\end{equation}
Use $y=(\rho_p,R,u_1,u_2,\psi,z_1,z_2)$ and
\begin{equation}\label{eq:nest-coordinates}
 \begin{aligned}
 S&=R-(1-\gamma)tu_1,&\psi_0&=\psi-(1-\gamma)tz_1,\\
 \rho_a&=S-\beta t\eta u_2,&\varphi_a&=\psi_0-\beta t\eta z_2,\\
 \rho_b&=S+\alpha t\eta u_2,&\varphi_b&=\psi_0+\alpha t\eta z_2,\\
 \rho_c&=R+\gamma tu_1,&\varphi_c&=\psi+\gamma tz_1,\qquad\varphi_p=0.
 \end{aligned}
\end{equation}
For either row family $P=R,T$, define
\begin{equation}\label{eq:nest-rows}
 P_0=\alpha P_a+\beta P_b,\quad
 A_P=\gamma P_0+(1-\gamma)P_c,\quad
 B_P=(P_c-P_0)/t,\quad C_P=(P_b-P_a)/(t\eta).
\end{equation}
Retain $(R_p,A_R,B_R,C_R,A_T,B_T,C_T)$. The inverse is
\begin{equation}\label{eq:nest-recovery}
 \begin{aligned}
 P_0&=A_P-(1-\gamma)tB_P,&P_c&=A_P+\gamma tB_P,\\
 P_a&=P_0-\beta t\eta C_P,&P_b&=P_0+\alpha t\eta C_P.
 \end{aligned}
\end{equation}
The coordinate determinant has absolute value $t^4\eta^2$, and $(1-t^3)T_p+t^3A_T=0$ recovers the exterior row. No row is lost at $\alpha=0,1$.

To verify joint smoothness in the independent parameters, put $P_a=u_1+\beta\eta u_2$, $P_b=u_1-\alpha\eta u_2$, $Q_a=z_1+\beta\eta z_2$, $Q_b=z_1-\alpha\eta z_2$. Then $r_{ab}=t\eta W$, $r_{ic}=tV_i$ for $i=a,b$, where
\begin{equation}\label{eq:nest-distances}
 W^2=u_2^2+\rho_a\rho_bz_2^2s(t\eta z_2/2)^2,\qquad
 V_i^2=P_i^2+\rho_i\rho_cQ_i^2s(tQ_i/2)^2.
\end{equation}
Let $H,K$ be the exterior angular and radial fields,
\[
 \begin{aligned}
 H(\rho,v)&=-(1-t^3)\rho\rho_p\sin v
 (|\rho e(v)-\rho_pe(0)|^{-3}-c_\eps),\\
 K(\rho,v)&=g(\rho)-\eps(1-t^3)
 \frac{\rho-\rho_p\cos v}{|\rho e(v)-\rho_pe(0)|^3}-\eps c_\eps X\cdot e(v).
 \end{aligned}
\]
For target values write $H_i=H(\rho_i,\varphi_i)$, and similarly for $K$; index $0$ denotes the average with weights $\alpha,\beta$. Define
\[
 L_i=\rho_i\rho_cQ_i s(tQ_i)(V_i^{-3}-t^3c_\eps),\qquad
 M=\rho_a\rho_bz_2s(t\eta z_2)(W^{-3}-t^3\eta^3c_\eps).
\]
Exact substitution gives
\begin{equation}\label{eq:nest-angular}
 \begin{aligned}
 A_T&=\gamma H_0+(1-\gamma)H_c,\\
 B_T&=(H_c-H_0)/t-L_0,\\
 C_T&=(H_b-H_a)/(t\eta)+(1-\gamma)(L_b-L_a)/\eta-M.
 \end{aligned}
\end{equation}
The field differences use \eqref{eq:hadamard}. For the remaining quotient define
\[
 \mathcal L(r,P,Q)=r\rho_cQ s(tQ)
 [(P^2+r\rho_cQ^2s(tQ/2)^2)^{-3/2}-t^3c_\eps].
\]
Along $(r,P,Q)=(S+t\eta\lambda u_2,u_1-\eta\lambda u_2,z_1-\eta\lambda z_2)$,
\begin{equation}\label{eq:nest-kernel-difference}
 \frac{L_b-L_a}{\eta}
 =\int_{-\beta}^{\alpha}
 (tu_2\mathcal L_r-u_2\mathcal L_P-z_2\mathcal L_Q)(r,P,Q)\dd\lambda.
\end{equation}
The substitution $\lambda+\beta$ fixes the interval to $[0,1]$.

For the radial terms put
\[
 \begin{aligned}
 N_i&=(P_i-\rho_c\chi(t,Q_i))V_i^{-3},&
 O_i&=(P_i+\rho_i\chi(t,Q_i))V_i^{-3},\\
 N&=(u_2-\rho_b\chi(t\eta,z_2))W^{-3},&
 O&=(u_2+\rho_a\chi(t\eta,z_2))W^{-3},
 \end{aligned}
\]
and $J=-\eps t\eta\alpha\beta(\rho_a+\rho_b)\chi(t\eta,z_2)W^{-3}$. The exact rows are
\begin{equation}\label{eq:nest-radial}
 \begin{aligned}
 A_R&=\gamma K_0+(1-\gamma)K_c
       +\eps t\gamma(1-\gamma)(N_0-O_0)+\gamma J,\\
 B_R&=(K_c-K_0)/t-\eps[\gamma O_0+(1-\gamma)N_0]-J/t,\\
 C_R&=(K_b-K_a)/(t\eta)+\eps(1-\gamma)(N_b-N_a)/\eta
       -\eps(\alpha O+\beta N).
 \end{aligned}
\end{equation}
Here $J/t$ is explicitly smooth. The quotient $(N_b-N_a)/\eta$ has formula \eqref{eq:nest-kernel-difference} with
\[
 \mathcal N(r,P,Q)=(P-\rho_c\chi(t,Q))
 (P^2+r\rho_cQ^2s(tQ/2)^2)^{-3/2}.
\]
The exterior row has no internal denominator. Positive $W,V_a,V_b$ therefore give joint smoothness in $(\eps,t,\eta,\alpha)$, including all scale faces and both zero-mass endpoints.

At $\eps=t=\eta=0$ the radial rows are $g(\rho_p),g(R),g'(R)u_1,g'(R)u_2$. Put
\[
 L_0=\frac{R^2z_1}{(u_1^2+R^2z_1^2)^{3/2}},\qquad
 M_0=\frac{R^2z_2}{(u_2^2+R^2z_2^2)^{3/2}}.
\]
With $\mathcal H$ from \eqref{eq:Hzero}, the angular rows are
\begin{equation}\label{eq:nest-boundary}
 \begin{aligned}
 A_T&=\mathcal H(R,\psi;\rho_p),\\
 B_T&=\mathcal H_\rho u_1+\mathcal H_vz_1-L_0,\\
 C_T&=\mathcal H_\rho u_2+\mathcal H_vz_2-u_2\partial_{u_1}L_0-z_2\partial_{z_1}L_0-M_0.
 \end{aligned}
\end{equation}
They vanish at
\begin{equation}\label{eq:nest-base}
 y^0=(1,1,0,0,\tau h,Z_0,z_0),\qquad
 Z_0=\omega D_0^{-1/3},\quad z_0=(3D_0)^{-1/3}.
\end{equation}
The radial block is $(3I_4,0)$ and the angular block is
\begin{equation}\label{eq:nest-angle-J}
 \begin{pmatrix}
 D_0&0&0\\
 CZ_0&3D_0&0\\
 Cz_0&-6D_0z_0/Z_0&9D_0
 \end{pmatrix},\qquad C=-\tfrac{27}{4}\sin(\tau h).
\end{equation}
Indeed $\partial_{z_1}L_0=-2D_0$, $\partial_{z_1}^2L_0=6D_0/Z_0$ and $\partial_{z_2}M_0=-6D_0$ at the base. Hence $\det D_yF=2187D_0^3\ne0$. The base, derivative and positive denominator bounds are independent of $\alpha$. Lemma~\ref{lem:product} supplies one root and a bounded inverse on an independent $(\eps,t,\eta)$ product, uniformly for $\alpha\in[0,1]$. At $\eps=0$ row recovery gives unit radii and $u_1=u_2=0$; continuity covers the scale faces. Thus $u_j=O(\eps)$ uniformly. At positive parameters Lemma~\ref{lem:hessian} gives nondegeneracy.

For capture, take $t=\nu^{1/3}$, $\eta=(\xi/\nu)^{1/3}$, $\alpha=\mu_a/\xi$. Proposition~\ref{prop:nested-scale} gives
\[
 t,\eta\to0,\quad d/t\to D_0^{-1/3},\quad
 z_2=\delta/(t\eta)\to z_0,\quad
 u_2=(\rho_b-\rho_a)/(t\eta)=O(\eps)\to0.
\]
The larger-scale radial estimate gives $u_1\to0$. The two signed angular differences from the pair to $c$, divided by $d$, tend to $\omega$, so averaging gives $z_1\to Z_0$. Also $\rho_p,R\to1$ and $\psi\to\tau h$. The variables enter the common root ball without convergence of $\alpha$. Table~\ref{tab:nested} lists the four possibilities.

\begin{table}[htbp]
\centering
\caption{Nested triple charts in the standard directed order.}\label{tab:nested}
\begin{tabular}{cccc}
\hline
$E$&$\{a,b\}$&$p$&$(\tau,\omega)$\\\hline
$\{1,2,3\}$&$\{1,2\}$&$4$&$(-1,+1)$\\
$\{1,2,3\}$&$\{2,3\}$&$4$&$(-1,-1)$\\
$\{2,3,4\}$&$\{2,3\}$&$1$&$(+1,+1)$\\
$\{2,3,4\}$&$\{3,4\}$&$1$&$(+1,-1)$\\\hline
\end{tabular}
\end{table}

Extracting fixed labels selects one chart. Section~\ref{sec:convexity} proves strict convexity and identifies its fixed-mass endpoint at $\eps=0$.

\section{Strict convexity and fixed-mass endpoints}
\label{sec:convexity}
The triple charts require exact collinearity at $t=0$ to preserve convexity on an independent parameter product.

Set $a=(1+\eps)^{1/3}$, $e_0=e(\tau h)$, $w_0=e(0)-e_0$ and $q^0=ae_0$.
At $t=0$, the exterior is $q_p=ae(0)$ and the probe field is
\[
 F(q)=q-\frac q{|q|^3}
      +\eps\frac{q_p-q}{|q_p-q|^3}-\eps\frac{q_p}{a^3}.
\]
The equilateral background gives $F(q^0)=0$. In the basis $(e_0,e_0^\perp)$ its derivative is
\begin{equation}\label{eq:bg-matrix}
 B=\frac3{1+\eps}(e_0e_0^T+\eps w_0w_0^T)
  =\frac3{1+\eps}
   \begin{pmatrix}1+\eps/4&\tau\sqrt3\eps/4\\
                   \tau\sqrt3\eps/4&3\eps/4\end{pmatrix}.
\end{equation}
Its trace is $3$ and its smaller eigenvalue is
\begin{equation}\label{eq:slow}
 \lambda=\eps d,\qquad
 d=\frac9{2(1+\eps)^2\bigl(1+\sqrt{1-3\eps/(1+\eps)^2}\bigr)},
 \qquad d(0)=D_0.
\end{equation}
All quantities below depend smoothly on small $\eps$:
\[
 k=-\frac{B_{12}}{B_{11}-\lambda},\qquad
 n=\frac{ke_0+e_0^\perp}{\sqrt{1+k^2}},\qquad
 l=(D_0/d)^{1/3},\qquad b=\frac{l}{a\sqrt{1+k^2}}.
\]
Here $Bn=\lambda n$, $k(0)=0$ and $l(0)=b(0)=1$.
If distinct scalar positions $s_i$ satisfy
\begin{equation}\label{eq:scalar-slow}
 D_0s_i-\sum_{j\ne i}\omega_j\frac{s_i-s_j}{|s_i-s_j|^3}=0,
 \qquad \sum_i\omega_i s_i=0,\quad\sum_i\omega_i=1,
\end{equation}
then $b_i=ls_i n$ satisfy
\begin{equation}\label{eq:vector-slow}
 Bb_i-\eps\sum_{j\ne i}\omega_j\frac{b_i-b_j}{|b_i-b_j|^3}=0,
\end{equation}
since $dl=D_0l^{-2}$.

For fixed $\eps\ne0$ and, in the nested chart, $\eta>0$, substitute
$q_i(t)=q^0+tb_i+O(t^2)$ and $\mu_i=t^3\omega_i$ in the Cartesian equations, keeping $q_p=ae(0)$ fixed.
Their first-order terms are $t$ times the left side of \eqref{eq:vector-slow}; the exterior mass and reaction vector change only at order $t^3$.
On these paths $Y-c_\eps X=O(t^3)$, and the exact identities
\[
 R_i=[F_i+\eps(Y-c_\eps X)]\cdot e(\varphi_i),\qquad
 T_i=\frac{\rho_i}{\eps}[F_i+\eps(Y-c_\eps X)]\cdot e^\perp(\varphi_i)
\]
identify the boundary difference rows with the radial projections and $a/\eps$ times the tangential projections. Basis derivatives contribute nothing because $F(q^0)=0$; the center and exterior rows vanish on the background.
Thus \eqref{eq:vector-slow} solves these boundary equations. Their joint smooth extensions from Section~\ref{sec:charts} also give the identities at zero parameters.

For the regular chart take $s_i=z_i(0,0,\alpha)$ from its own root.
Weighted row recovery gives \eqref{eq:scalar-slow}. The resulting boundary root is
\begin{equation}\label{eq:regular-slow}
 r=R=a,\qquad \psi=\tau h,\qquad
 v=bv(0,0,\alpha),\qquad u=akv.
\end{equation}
The factor $a^{-1}$ in $b$ converts tangential to angular displacement.
This candidate stays in the chart's common neighborhood, so uniqueness identifies it with the root at $t=0$. In particular,
\begin{equation}\label{eq:regular-parallel}
 w_i=akz_i\qquad(t=0).
\end{equation}

For the nested chart, its root at $\eps=t=0$, with $\eta>0$ fixed, gives
\[
 s_a=-(1-\gamma)z_1-\beta\eta z_2,\quad
 s_b=-(1-\gamma)z_1+\alpha\eta z_2,\quad s_c=\gamma z_1,
\]
with weights $(\gamma\alpha,\gamma\beta,1-\gamma)$.
The angular difference rows make the three scalar residuals equal, and their weighted sum is zero. Hence \eqref{eq:scalar-slow} holds, including $\alpha=0,1$.
Set $\rho_p=R=a$, $\psi=\tau h$, multiply both $z_j$ by $b$, and set $u_j=akz_j$.
This boundary root is smooth in $\eta,\alpha$ and stays in the common neighborhood.
Uniqueness and continuity as $\eta\downarrow0$ yield
\begin{equation}\label{eq:nest-parallel}
 u_2z_1-u_1z_2=0\qquad(t=0)
\end{equation}
for every small $\eps$ and all $\alpha\in[0,1]$.

We use these collinearity identities to prove strict convexity.
For $A<B<C<A+\pi$,
\begin{equation}\label{eq:circle-area}
 \det(e(B)-e(A),e(C)-e(A))
 =4\sin\frac{B-A}2\sin\frac{C-B}2\sin\frac{C-A}2>0.
\end{equation}
With the origin as third vertex, the corresponding area is $\sin(B-A)>0$.

In the regular chart, line-segment integration gives smooth divided edges
\[
 q_j-q_i=tV_{ij},\qquad
 V_{ij}|_{t=0}=(w_j-w_i)e_0+R(z_j-z_i)e_0^\perp.
\]
Evaluate $B_1=\det(V_{12},V_{13})$ along the smooth chart root.
Equation~\eqref{eq:regular-parallel} gives $B_1(\eps,0,\alpha)=0$, so
\begin{equation}\label{eq:regular-area}
 \det(q_2-q_1,q_3-q_1)=t^3C_1,\qquad
 C_1(\eps,t,\alpha)=\int_0^1\partial_tB_1(\eps,st,\alpha)\dd s.
\end{equation}
The derivative includes the root's dependence on $t$.
At $\eps=0$ every radius is one; writing $v^0=v(0,0,\alpha)$, \eqref{eq:circle-area} gives
\[
 C_1(0,0,\alpha)=\tfrac12v^0_1v^0_2(v^0_1+v^0_2)>0.
\]
Thus $C_1$ stays positive on a smaller product.
Areas containing both exterior vertices have positive equilateral limits.
Those containing two cluster members and one exterior vertex have a factor $t$; at the base their quotients are $z_j-z_i$ for the origin and $(z_j-z_i)(1-\cos h)$ for $p$, with cyclic order understood. They are positive.

In the nested chart write
\[
 q_b-q_a=t\eta V_{ab},\qquad q_c-q_a=tV_{ac},\qquad
 B_2=\det(V_{ab},V_{ac})
\]
along its root. At $t=0$,
\[
 B_2=R\{u_2(z_1+\beta\eta z_2)-z_2(u_1+\beta\eta u_2)\}
    =R(u_2z_1-u_1z_2)=0.
\]
Hence
\begin{equation}\label{eq:nest-area}
 \det(q_b-q_a,q_c-q_a)=t^3\eta C_2,\qquad
 C_2=\int_0^1\partial_tB_2(\eps,st,\eta,\alpha)\dd s.
\end{equation}
Again the derivative is along the root. Equation~\eqref{eq:circle-area} gives
\[
 C_2(0,0,0,\alpha)=\tfrac12z_0Z_0^2>0.
\]
This is independent of $\alpha$ and positive for either sign of $Z_0$; $(c,a,b)$ and $(a,b,c)$ are cyclic permutations.
Table~\ref{tab:areas} gives the other area quotients at $\eps=t=\eta=0$, in cyclic order.

\begin{table}[htbp]
\centering
\caption{Remaining signed-area quotients in the nested chart.}\label{tab:areas}
\begin{tabular}{ccc}
\hline
 Vertices&Factor&Limiting quotient\\\hline
 $0,p$, one cluster member&$1$&$\sin h$\\
 $0,a,b$&$t\eta$&$z_0$\\
 $p,a,b$&$t\eta$&$z_0(1-\cos h)$\\
 $0,c$, one pair member&$t$&$|Z_0|$\\
 $p,c$, one pair member&$t$&$|Z_0|(1-\cos h)$\\\hline
\end{tabular}
\end{table}

Continuity preserves all ten cyclic triple areas on a common product, uniformly for $\alpha\in[0,1]$.
Thus every consecutive edge has the other vertices strictly on its left, proving strict convexity in both triple charts.

\begin{proposition}\label{prop:chart-interface}
After restriction to smaller independent parameter products, all four chart types capture their actual sequences. Their roots with positive satellite masses and $\eps>0$ are strictly convex in the prescribed order and have nondegenerate shape Hessian.
Fixing any positive mass vector represented by a chart and varying only $\eps$ gives a smooth branch on $[0,\bar\eps)$. At zero it has unit radii and distinct coorbital angles in the same open semicircle, for that same mass vector. For a captured sequence, $\bar\eps>0$ is independent of its index after extraction into one chart.
\end{proposition}
\begin{proof}
Smoothness, capture and nondegeneracy follow from Section~\ref{sec:charts}; convexity follows from the area factors above and the pair and noncollision estimates there.
With positive mass coordinates fixed, the scale parameters $t$ and, when present, $\eta$ are positive. The row transformations remain invertible at $\eps=0$: radial recovery gives $g(\rho_i)=0$, hence $\rho_i=1$, and the weighted angular identity recovers the omitted row.
The pair gap $tz$, regular gaps $tv_i$, and nested gaps $t\eta z_2,tQ_a,tQ_b$ remain nonzero in their prescribed order. Exterior separations and span $<\pi$ persist.
Consequently the endpoint is a noncollision coorbital root for the fixed positive masses.
Each parameter restriction is made once about the selected base; freezing the mass coordinates therefore leaves a common $\eps$ interval for the captured sequence.
\end{proof}

\section{Connection of all solutions and the common threshold}
\label{sec:global}
Let $\mathcal S(\eps,\mu)$ consist of the strictly convex multiplier-one representatives in standard directed order,
\[
 y=(\rho_1,\rho_2,\rho_3,\rho_4,\varphi_2,\varphi_3,\varphi_4),
 \qquad \varphi_1=0.
\]
Each allowed similarity class has exactly one such representative.

\begin{proposition}\label{prop:connection}
If $y_n\in\mathcal S(\eps_n,\mu^{(n)})$ and $\eps_n\to0$, then a subsequence admits smooth branches $\Gamma_n:[0,\bar\eps)\to\R^7$, with $\bar\eps>0$ independent of $n$, such that $\Gamma_n(\eps_n)=y_n$ and
\[
 \Gamma_n(s)\in\mathcal S(s,\mu^{(n)})\qquad(0<s<\bar\eps).
\]
Their shape Hessians are nondegenerate for $s>0$, and their zero endpoints are distinct-angle coorbital roots for $\mu^{(n)}$ in the same order.
\end{proposition}
\begin{proof}
Proposition~\ref{prop:uniform} gives unit-radius limits and bounds the span away from both zero and $\pi$.
Extract limits of the angles and masses. Equal-angle blocks are consecutive; a four-member block is impossible, and Proposition~\ref{prop:pair-scale} excludes two pairs.
For a triple, extract the shortest-to-largest distance ratio. A positive limit gives the regular chart, after extracting a distinct scalar root; a zero limit gives the nested chart, after fixing the shortest pair.
Together with noncollision and one-pair limits, these exhaust the possibilities.
The charts capture each case, and Proposition~\ref{prop:chart-interface} gives branches with their mass coordinates fixed on one common $\eps$ interval. The pair fractions need not converge because the corresponding charts are uniform on $[0,1]$.
Subtracting the current $\varphi_1$ from every angle restores the required gauge smoothly, preserves the equations and nondegeneracy, and returns $y_n$ at $\eps_n$.
\end{proof}

\begin{proof}[Proof of Theorem~\ref{thm:main}]
For each fixed $\mu\in\M$, Proposition~\ref{prop:local} gives a locally unique convex branch from
$y^0(\mu)=(1,1,1,1,\theta_2^0(\mu),\theta_3^0(\mu),\theta_4^0(\mu))$ on $0\le s<\ell(\mu)$.
Theorem~\ref{thm:coorbital} identifies every endpoint in Proposition~\ref{prop:connection} with this point for its fixed mass vector.

If there were no common existence bound, we could choose $\eps_n\to0$ and $\mu^{(n)}$ with $\mathcal S(\eps_n,\mu^{(n)})$ empty.
Choose $0<s_n<\frac12\min\{\eps_n,1/n,\ell(\mu^{(n)})\}$ and take the local root at $s_n$.
Proposition~\ref{prop:connection} gives branches through these roots on a common interval containing $\eps_n$ for large $n$, a contradiction. Let $\delta_1>0$ be a common existence bound.

If there were no common uniqueness bound, there would be two distinct roots for the same $\mu^{(n)}$ at $\eps_n\to0$.
Apply Proposition~\ref{prop:connection} successively to the two sequences. On a common subsequence their branches $\Gamma_n^1,\Gamma_n^2$ share an interval $[0,\tau)$.
Their endpoints coincide, so fixed-mass local uniqueness makes the branches equal near zero.
For each sufficiently large $n$, the set
\[
 Q_n=\{s\in[0,\tau):\Gamma_n^1(s)=\Gamma_n^2(s)\}
\]
is closed and contains a neighborhood of zero. At every positive point it is open by nondegeneracy, Lemma~\ref{lem:hessian} and the implicit function theorem.
Connectedness gives $Q_n=[0,\tau)$, contradicting distinctness at $\eps_n$. This yields $\delta_2>0$.

A sequence of degenerate roots with $\eps_n\to0$ would be captured by Proposition~\ref{prop:connection}, contradicting nondegeneracy along its positive-parameter branches. Thus there is $\delta_3>0$.

Set $\delta_0=\min\{\delta_1,\delta_2,\delta_3\}$.
Simultaneous relabeling of positions and masses takes each directed order to the standard order and maps $\M$ onto itself, so the same bound applies to every order.
Undoing the multiplier-one normalization proves the theorem.
\end{proof}

\end{document}